\documentclass{article}
\usepackage{amsmath}
\usepackage{amsfonts}
\usepackage{graphicx}
\usepackage{graphics}

\newtheorem{theorem}{Theorem}

\newtheorem{conjecture}{Conjecture}
\newtheorem{corollary}{Corollary}

\newtheorem{definition}{Definition}
\newtheorem{example}{Example}

\newtheorem{lemma}{Lemma}

\newenvironment{proof}[1][Proof]{\noindent\textbf{#1.} }{\ \rule{0.5em}{0.5em}}
\begin{document}

\title{Eigenvalues of a $H$-generalized join graph operation constrained by vertex subsets}
\author{Domingos M. Cardoso and Enide A. Martins\thanks{
Work supported by {\it FEDER} founds through {\it COMPETE}--Operational Programme Factors of Competitiveness
and by Portuguese founds through the {\it Center for Research and Development in Mathematics and Applications} (University of Aveiro)
and the Portuguese Foundation for Science and Technology (``FCT--Funda\c{c}\~{a}o para a Ci\^{e}ncia e a Tecnologia''), within project
PEst-C/MAT/UI4106/2011 with COMPETE number FCOMP-01-0124-FEDER-022690 and also to the project PTDC/MAT/112276/2009. These authors also thanks the
hospitality of Departamento de Matem\'aticas, Universidad Cat\'olica del Norte, Chile, during the visit of which this research was
finalized.} \\
Departamento de Matem\'{a}tica\\
Universidade de Aveiro\\
Aveiro, Portugal
\and Maria Robbiano$^{\S}$\thanks{%
Research partially supported by Fondecyt - IC Project 11090211, Chile.} and
Oscar Rojo$^{\S}$\thanks{Research supported by Project Fondecyt 1100072, Chile. $^{\S}$These authors thanks the hospitality of
Departamento de Matem\'{a}tica, Universidade de Aveiro, Aveiro, Portugal, in which this research was started.} \\
Departamento de Matem\'aticas\\
Universidad Cat\'{o}lica del Norte\\
Antofagasta, Chile}
\date{}
\maketitle

\begin{abstract}
Considering a graph $H$ of order $p$, a generalized $H$-join operation of a family of graphs $G_1, \ldots, G_p$, constrained
by a family of vertex subsets $S_i \subseteq V(G_i)$, $i=1, \ldots, p,$ is introduced. When each vertex subset $S_i$ is
$(k_i,\tau_i)$-regular, it is deduced that all non-main adjacency eigenvalues of $G_i$, different from $k_i-\tau_i$, for $i=1, \ldots, p,$
remain as eigenvalues of the graph $G$ obtained by the above mentioned operation. Furthermore, if each graph $G_i$ of the
family is $k_i$-regular, for $i=1, \ldots, p$, and all the vertex subsets are such that $S_i=V(G_i)$, the $H$-generalized
join operation constrained by these vertex subsets coincides with the $H$-generalized join operation. Some applications on the
spread of graphs are presented. Namely, new lower and upper bounds are deduced and a infinity family of non regular graphs of order $n$
with spread equals $n$ is introduced.
\end{abstract}



\section{Notation and main concepts}
We deal with undirected simple graphs herein simply called graphs. For each graph $G$, the vertex set is denoted by $V(G)$ and its edge set by $E(G).$
Usually, we consider that the graph $G$ has order $n$, that is $V(G)=\{1, \ldots, n\}$. An edge with end vertices $i$ and $j$ is denoted by $ij$ and
then we say that the vertices $i$ and $j$ are adjacent or neighbors. The number of neighbors of a vertex $i$ is the degree of $i$ and the neighborhood
of $i$ is the set of its neighbors, $N_G(i)=\{j \in V(G): ij \in E(G)\}$. The maximum and minimum degree of the vertices of $G$ is denoted by $\Delta(G)$
and $\delta(G)$, respectively. The complement of $G$, denoted by $\overline{G}$ is such that $V(\overline{G})=V(G)$ and
$E(\overline{G})=\{ij: ij \not \in E(\overline{G})\}$. A path of length $p-1$, $P_p$, in $G$ is a sequence of vertices $i_1, \ldots, i_p$ all distinct
except, eventually the first and the last) and such that $i_ji_{j+1} \in E(G)$, for $j=1, \ldots, p-1$. If $i_1=i_p$, then it is a closed path usually
called cycle of length $p$ and denoted $C_p$.
A graph $G$ is connected if there is a path between each pair of distinct vertices. A complete
graph of order $n$, where each pair of distinct vertices is connected by an edge, is denote by $K_n$. The complement of $K_n$, $\overline{K}_n$,
is called the null graph. A graph $G$ is bipartite if $V(G)$ can be partitioned into two subsets $V_1$ and $V_2$ such that every edge of $G$ has
one end vertex in $V_1$ and the other one in $V_2$. This graph $G$ is called complete bipartite and it is denoted $K_{p,q}$, if $|V_1|=p$,
$|V_2|=q$ and each vertex of $V_1$ is connected with every vertex of $V_2$.\\
The adjacency matrix of a graph $G$, $A\left(G\right)=\left(a_{i,j}\right)$, is the $n\times n$ matrix
\begin{eqnarray*}
a_{i,j} &=& \left\{\begin{array}{cl}
                         1 & \text{if } ij \in E(G) \\
                         0 & \text{otherwise.}
                   \end{array}%
            \right.
\end{eqnarray*}%
Then the matrix $A\left(G\right)$ is a nonnegative symmetric with entries which are $0$ and $1$ and then all of its eigenvalues are real.
Furthermore, since all its diagonal entries are equal to $0$, the trace of $A(G)$ is zero. If $G$ has at least one edge, then
$A\left(G\right)$ has a negative eigenvalue not greater than $-1$ and a positive eigenvalue not less than the average degree
of the vertices of $G$. Considering any matrix $M$ we denote its spectrum (the multiset of the eigenvalues of $M$) by $\sigma(M)$.
The spectrum of the adjacency matrix of a graph $G$, $\sigma(A(G))$, is simply denoted  by $\sigma(G)$ and the eigenvalues of $A(G)$
are also called the eigenvalues of $G$. An eigenvalue $\lambda$ of a graph $G$ is called non-main if its associated
eigenspace, denoted $\varepsilon_{G}(\lambda),$ is orthogonal to the all one vector, otherwise is called main.\\
Usually, the multiplicities of the eigenvalues are represented in the multiset $\sigma(G)$ as powers in square
brackets. For instance, $\sigma(G)=\{\lambda_1^{[m_1]}, \ldots, \lambda_q^{[m_q]}\}$ denotes that $\lambda_1$
has multiplicity $m_1$, $\lambda_2$ has multiplicity $m_2$, and so on. Throughout the paper, the eigenvalues of a graph $G$
with $n$ vertices, $\lambda_1(G), \ldots, \lambda_n(G)$, are ordered as follows: $\lambda_1(G) \ge \cdots \ge \lambda_n(G)$.
If $\lambda$ is an eigenvalue of the graph $G$ and $u$ is an associated eigenvector, the pair $(\lambda,u)$ is called an eigenpair of $G$.\\
Considering a graph $G$ of order $n$ and a vertex subset $S \subseteq V(G)$, the characteristic vector of $S$ is the vector $x_S \in \{0,1\}^n$
such that $(x_S)_v = \left\{\begin{array}{ll}
                                   1 & \hbox{if } v \in S\\
                                   0 & \hbox{otherwise.}
                            \end{array}
                     \right.$
A vertex subset $S$ is $(k,\tau)$-regular if $S$ induces a $k$-regular graph in $G$ and every vertex out of $S$ has $\tau$ neighbors  in $S$, that is,
$$
|N_G(i) \cap S| = \left\{\begin{array}{ll}
                           k    & \hbox{if } i \in S\\
                           \tau & \hbox{otherwise.}
                         \end{array}
                  \right.
$$
When the graph $G$ is $k$-regular, for convenience, $S=V(G)$ is considered $(k,0)$-regular. There are several properties
of graphs related with $(k,\tau)$-regular sets (see \cite{CSZ2008, CR2007}). For instance, we may refer the following properties:
\begin{itemize}
\item A graph $G$ has a perfect matching if and only if its line graph has a $(0,2)$-regular set.
\item A graph $G$ is Hamiltonian if and only if its line graph has a $(2,4)$-regular set inducing a connected graph.
\item A graph $G$ of order $n$ is strongly regular with parameters $(n,p,a,b)$ if and only if $\forall v \in V(G)$ the vertex subset $S=N_G(v)$
      is $(a,b)$-regular in $G-v$ (where $G-v$ is the graph obtained from $G$ deleting the vertex $v$).
\end{itemize}

\section{Generalized join graph operation with vertex subset constraints}

Considering two vertex disjoint graphs $G_1$ and $G_2$, the join of $G_1$ and $G_2$ is the graph $G_1 \vee G_2$ such that
$V(G_1 \vee G_2)=V(G_1) \cup V(G_2)$ and $E(G_1 \vee G_2)=E(G_1) \cup E(G_2) \cup \{xy: x \in V(G_1) \wedge y \in V(G_2)\}$.
A generalization of the join operation was first introduced in \cite{Schwenk74} under the designation of \textit{generalized composition}
and more recently in \cite{CFMR2011} with the designation of \textit{$H$-join}, defined as follows:

Consider a family of $p$ graphs, $\mathcal{F}=\{G_1, \ldots, G_p\}$, where each graph $G_j$ has order $n_j$, for $j=1, \ldots, p,$
and a graph $H$ such that $V(H)=\{1, \ldots, p\}$. Each vertex $j \in V(H)$ is assigned to the graph $G_j \in \mathcal{F}$.
The $H$-join (generalized composition) of $G_{1},\ldots ,G_p$ is the graph
$G=\bigvee_{H}{\{G_{j}: j \in V(H)\}}$ ($H[G_1, \ldots, G_p]$) such that $V(G)=\bigcup_{j=1}^{p}{V(G_j)}$
and
\begin{equation*}
E(G)=\left( \bigcup_{j=1}^{p}{E(G_j)}\right) \cup \left( \bigcup_{rs\in E(H)}{\{uv:u\in V(G_{r}),v\in V(G_{s})\}}\right) .
\end{equation*}%
Now, we generalize the above $H$-join operation according to the next definition.

\begin{definition}\label{generalized_join_definition}
Consider a graph $H$ of order $p$ and a family of $p$ graphs $\mathcal{F}=\{G_1, \ldots, G_p\}$. Consider also a family of vertex subsets
$\mathcal{S}=\{S_1, \ldots, S_p\}$, such that $S_i \subseteq V(G_i)$ for $i=1, \ldots, p$. The $H$-generalized join operation of the family
of graphs $\mathcal{F}$ constrained by the family of vertex subsets $\mathcal{S}$, denoted by $\bigvee_{(H,\mathcal{S})}\mathcal{F}$,
produces a graph such that
\begin{eqnarray*}
V\left(\bigvee_{(H,\mathcal{S})}\mathcal{F}\right) &=& \bigcup_{i=1}^{p}{V(G_i)},\\
E\left(\bigvee_{(H,\mathcal{S})}\mathcal{F}\right) &=& \left(\bigcup_{i=1}^{p}{E(G_i)}\right) \cup \{xy: x \in S_i, y \in S_j, ij \in E(H)\}.
\end{eqnarray*}
\end{definition}

Notice that the particular case of the $H$-generalized join operation of the family of graphs $\mathcal{F}=\{G_1, \ldots, G_p\}$ constrained
by the family of vertex subsets $\mathcal{S}=\{V(G_1), \ldots, V(G_p)\}$, coincides with the above described $H$-generalized join operation.

\begin{example}\label{ex1}
The Figure~\ref{figura1} depicts an example of a $H$-generalized join operation, with $H=P_3$, of a family of graphs $\mathcal{F}=\{G_1, G_2, G_3\}$
constrained by the family of vertex subsets $\mathcal{S}=\{S_1, S_2, S_3\}$, where $S_1=\{a,b\},$ $S_2=\{d,f\},$ and $S_3=\{g,i,j\}.$
\begin{figure}[th]
\begin{center}
\unitlength=0.4mm
\begin{picture}(150,150)(100,-60)
%
%
%
\put(135,80){\circle*{4}} 
\put(135,87){\makebox(0,0){\small 1}}
\put(180,80){\circle*{4}} 
\put(180,87){\makebox(0,0){\small 2}}
\put(225,80){\circle*{4}} 
\put(225,87){\makebox(0,0){\small 3}}
\put(135,80){\line(1,0){90}} 
\put(180,70){\makebox(0,0){\small $H=P_3$}}
%
%
\put(105,25){\circle*{4}} 
\put(100,25){\makebox(0,0){\small a}}
\put(135,40){\circle*{4}} 
\put(135,47){\makebox(0,0){\small b}}
\put(135,10){\circle*{4}} 
\put(135,3){\makebox(0,0){\small c}}
\put(105,25){\line(2,1){30}} 
\put(105,25){\line(2,-1){30}} 
\put(135,-5){\makebox(0,0){\small $G_1$}}
%
%
\put(180,40){\circle*{4}} 
\put(180,47){\makebox(0,0){\small d}}
\put(180,25){\circle*{4}} 
\put(187,25){\makebox(0,0){\small e}}
\put(180,10){\circle*{4}} 
\put(180,3){\makebox(0,0){\small f}}
\put(180,40){\line(0,-1){30}} 
\put(180,-5){\makebox(0,0){\small $G_2$}}
%
%
\put(210,25){\circle*{4}} 
\put(217,25){\makebox(0,0){\small j}}
\put(225,40){\circle*{4}} 
\put(225,47){\makebox(0,0){\small i}}
\put(240,25){\circle*{4}} 
\put(247,25){\makebox(0,0){\small h}}
\put(225,10){\circle*{4}} 
\put(225,3){\makebox(0,0){\small g}}
%
\put(225,10){\line(1,1){15}}  
\put(225,40){\line(-1,-1){15}}
\put(210,25){\line(1,-1){15}} 
\put(225,40){\line(0,-1){30}}  
\put(240,25){\line(-1,1){15}} 
\put(225,-5){\makebox(0,0){\small $G_3$}}
%
%
\put(105,-45){\circle*{4}} 
\put(100,-45){\makebox(0,0){\small a}}
\put(135,-30){\circle*{4}} 
\put(135,-23){\makebox(0,0){\small b}}
\put(135,-60){\circle*{4}} 
\put(135,-67){\makebox(0,0){\small c}}
\put(105,-45){\line(2,1){30}} 
\put(105,-45){\line(2,-1){30}} 
%
\put(180,-30){\circle*{4}} 
\put(180,-23){\makebox(0,0){\small d}}
\put(180,-45){\circle*{4}} 
\put(187,-45){\makebox(0,0){\small e}}
\put(180,-60){\circle*{4}} 
\put(180,-67){\makebox(0,0){\small f}}
\put(180,-30){\line(0,-1){30}} 
%
\put(225,-60){\circle*{4}} 
\put(225,-67){\makebox(0,0){\small g}}
\put(240,-45){\circle*{4}} 
\put(247,-45){\makebox(0,0){\small h}}
\put(225,-30){\circle*{4}} 
\put(225,-23){\makebox(0,0){\small i}}
\put(210,-45){\circle*{4}} 
\put(217,-45){\makebox(0,0){\small j}}
\put(240,-45){\line(-1,1){15}} 
\put(225,-60){\line(1,1){15}}  
\put(225,-30){\line(-1,-1){15}}
\put(210,-45){\line(1,-1){15}} 
\put(225,-60){\line(0,1){30}}  
%
\put(135,-30){\line(1,0){45}} 
\put(135,-30){\line(3,-2){45}}
\put(180,-60){\line(3,2){45}} 
\put(105,-45){\line(5,1){75}} 
\put(105,-45){\line(5,-1){75}}
\put(180,-30){\line(3,-2){45}} 
\put(180,-30){\line(1,0){45}} 
\put(180,-60){\line(1,0){45}} 
\put(180,-30){\line(2,-1){30}} 
\put(180,-60){\line(2,1){30}} 
\put(175,-75){\makebox(0,0){\small $\bigvee_{(H,\mathcal{S})}\mathcal{F}$}}
\end{picture}
\end{center}
\caption{The $H$-generalized join operation of the family of graphs $\mathcal{F}=\{G_1, G_2, G_3\}$, constrained by the family of vertex subsets
         $\mathcal{S}=\{S_1, S_2, S_3\}$, where $S_1=\{a,b\} \subset V(G_1), S_2=\{d,f\} \subset V(G_2)$ and $S_3=\{g,i,j\} \subset V(G_3)$.}\label{figura1}
\end{figure}
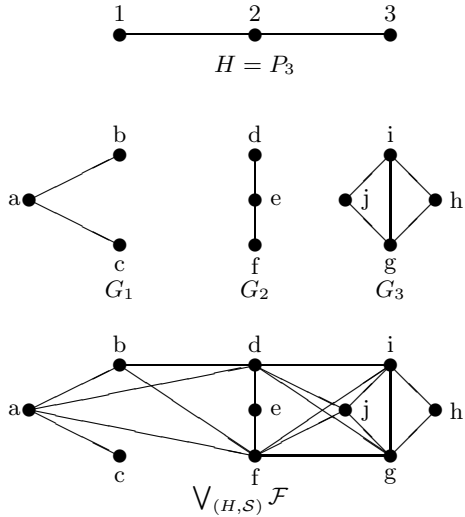
\end{example}

Now it is worth to recall the following result.
\begin{lemma}\cite{CSZ2008}\label{non-main}
Let $G$ be a graph with a $(\kappa ,\tau )$-regular set $S$, where $\tau > 0$, and $\lambda \in \sigma(A(G))$. Then, denoting the characteristic
vector of $S$ by $x_S$, $\lambda$ is non-main if and only if
$$
\lambda =\kappa - \tau  \qquad \text{ or } \qquad {\bf x}_S\in \left ({\mathcal E}_G(\lambda )\right )^{\bot},
$$
where $\left ({\mathcal E}_G(\lambda )\right )^{\bot}$ denotes the vector space orthogonal to the eigenspace associated to the eigenvalue $\lambda$.
\end{lemma}

From now on, given a graph $H$, we denote
$$
\delta_{i,j}(H) = \left\{\begin{array}{ll}
                          1 & \hbox{if } ij \in E(H) \\
                          0 & \hbox{otherwise.}
                   \end{array}
                          \right.
$$

\begin{theorem}\label{nom-main-eigenvalues-join}
Consider a graph $H$ of order $p$ and a family of $p$ graphs $\mathcal{F}=\{G_1, \ldots, G_p\}$ such that $|V(G_i)|=n_i, i=1, \ldots, p$.
Consider also the family of vertex subsets $\mathcal{S}=\{S_1, \ldots, S_p\}$, where
$$
S_i \in \{S'_i \subseteq V(G_i):  \text{ either } S'_i \text{ or } V(G_i) \setminus S'_i \text{ is } (k_i,\tau_i)-\text{regular for some integers } k_i, \tau_i\},
$$
for $i=1, \ldots, p$. Let $G=\bigvee_{(H,\mathcal{S})}\mathcal{F}$.
If $\lambda \in \sigma(G_i) \setminus \{k_i-\tau_i\}$ for some $i \in \{1, \ldots, p\}$ is non-main, then $\lambda \in \sigma(G).$
\end{theorem}

\begin{proof}
Denoting $\delta_{i,j} = \delta_{i,j}(H)$, then $\delta_{ij} x_{S_i}x^T_{S_j}$, where $x_{S_i}$ and $x_{S_j}$ are the characteristic
vectors of $S_i$ and $S_j$, respectively, is an $n_i \times n_j$ matrix whose entries are zero if $ij \not \in E(H)$, otherwise
$$
\left(\delta_{i,j} x_{S_i}x^T_{S_j}\right)_{q,r} =  \left\{\begin{array}{ll}
                                                      1 & \hbox{ if } q \in S_i \wedge r \in S_j\\
                                                      0 & \hbox{ otherwise.}
                                                      \end{array}
                                                    \right.,
$$
Then the adjacency matrix of $G$ has the form
$$
A(G) = \left(\begin{array}{ccccc}
       A(G_1)                             & \delta_{1,2}x_{S_1} x^T_{S_2}      & \cdots & \delta_{1,p-1}x_{S_1} x^T_{S_{p-1}}&\delta_{1,p}x_{S_1} x^T_{S_p}\\
       \delta_{2,1}x_{S_2} x^T_{S_1}        & A(G_2)                           & \cdots & \delta_{2,p-1}x_{S_2} x^T_{S_{p-1}}&\delta_{2,p}x_{S_2} x^T_{S_p}\\
       \delta_{3,1}x_{S_3} x^T_{S_1}        &\delta_{3,2}x_{S_3} x^T_{S_2}     & \cdots & \delta_{3,p-1}x_{S_3} x^T_{S_{p-1}}&\delta_{3,p}x_{S_3} x^T_{S_p}\\
       \vdots                              &   \vdots                          & \ddots & \vdots                              & \vdots \\
       \delta_{p-1,1}x_{S_{p-1}}x^T_{S_1} &\delta_{p-1,2}x_{S_{p-1}} x^T_{S_2} & \cdots & A(G_{p-1})                         &\delta_{p-1,p}x_{S_{p-1}} x^T_{S_p}\\
       \delta_{p,1}x_{S_p}x^T_{S_1}         &\delta_{p,2}x_{S_p} x^T_{S_2}     & \cdots &\delta_{p-1,p}x_{S_{p-1}} x^T_{S_p} & A(G_p)\\
  \end{array}\right).
$$
Let $u_i$ be an eigenvector of $A(G_i)$ associated to the non-main eigenvalue $\lambda_i \ne k_i-\tau_i$, with $1 \le i \le p$. Then,
\begin{equation}\label{eigenvalue-equation}
A(G) \left(\begin{array}{c}
                 0 \\
                 \vdots \\
                 0 \\
                 u_i \\
                 0 \\
                 \vdots \\
                 0
        \end{array}\right) = \left(\begin{array}{c}
                                    \delta_{1,i}\left(x^T_{S_i}u_i\right)x_{S_1}  \\
                                    \vdots \\
                                    \delta_{i-1,i}\left(x^T_{S_i}u_i\right)x_{S_{i-1}} \\
                                    A(G_i)u_i \\
                                    \delta_{i+1,i}\left(x^T_{S_i}u_i\right)x_{S_{i+1}} \\
                                    \vdots \\
                                    \delta_{p,i}\left(x^T_{S_i}u_i\right)x_{S_p}
                             \end{array}\right) = \lambda_i \left(\begin{array}{c}
                                                                   0 \\
                                                                   \vdots \\
                                                                   0 \\
                                                                   u_i \\
                                                                   0 \\
                                                                   \vdots \\
                                                                   0
                                                                  \end{array}
                                                              \right),
\end{equation}
since $x_{S_i}$ is the characteristic vector of the vertex subset $S_i$ and $S_i$ or $V(G_i) \setminus S_i$ is  $(k_i,\tau_i)$-regular
(take into account that $\lambda_i$ is non-main and then we may apply Lemma~\ref{non-main}).
\end{proof}

From the proof of Theorem~\ref{nom-main-eigenvalues-join}, we may conclude the following corollary.

\begin{corollary}
Consider a graph $H$ of order $p$ and a family of $p$ graphs $\mathcal{F}=\{G_1, \ldots, G_p\}$ such that $|V(G_i)|=n_i, i=1, \ldots, p$.
Consider also the family of vertex subsets $\mathcal{S}=\{V(G_1), \ldots, V(G_p)\}$. Let $G=\bigvee_{(H,\mathcal{S})}\mathcal{F}$.
If $\lambda \in \sigma(G_i)$ for some $i \in \{1, \ldots, p\}$ is non-main, then $\lambda \in \sigma(G).$
\end{corollary}

\begin{proof}
Consider an eigenpair $(\lambda,u)$ of a graph $G_i$, for some $i \in \{1, \ldots, p\}$, where $\lambda$ is non-main. Then,
taking into account the equations \eqref{eigenvalue-equation} where, in this case, $x_{S_i}$ is the all one vector, the result
follows.
\end{proof}

Notice that in the above corollary, $G=\bigvee_{(H,\mathcal{S})}\mathcal{F}$ coincides with the $H$-join operation  of the family of graphs
$\mathcal{F}$ \cite{CFMR2011} (generalized composition $H[G_1, \ldots,G_p]$ in the terminology of \cite{Schwenk74}).

\begin{example}\label{ex2}
Consider the Example~\ref{ex1}, where $V(G_1)=\{a,b,c\}$ and $S_1=\{a,b\}$ is $(1,1)$-regular, $V(G_2)=\{d,e,f\}$ and $S_2=\{d,f\}$ is
$(0,2)$-regular, $V(G_3)=\{g,h,i,j\}$ and $S_3=\{g,i,j\}$ is $(2,2)$-regular.
\begin{itemize}
\item The eigenpairs of $A(G_1)$ are {\small $\left(\sqrt{2},\left[\begin{array}{c}
                                                    \sqrt{2}\\
                                                    1 \\
                                                    1 \\
                                              \end{array}
                                       \right]\right)$, $\left(0,\left[\begin{array}{c}
                                                                              0\\
                                                                              -1 \\
                                                                              1 \\
                                                                         \end{array}
                                                                         \right]\right)$} and {\small $\left(-\sqrt{2},\left[\begin{array}{c}
                                                                                                                           -\sqrt{2}\\
                                                                                                                           1 \\
                                                                                                                           1 \\
                                                                                                                    \end{array}
                                                                                                                    \right]\right)$}.
\item The eigenpairs of $A(G_2)$ are {\small $\left(\sqrt{2},\left[\begin{array}{c}
                                                    \sqrt{2}\\
                                                    2 \\
                                                    \sqrt{2} \\
                                              \end{array}
                                       \right]\right)$, $\left(0,\left[\begin{array}{c}
                                                                              -1\\
                                                                              0 \\
                                                                              1 \\
                                                                         \end{array}
                                                                         \right]\right)$} and {\small $\left(-\sqrt{2},\left[\begin{array}{c}
                                                                                                                           1\\
                                                                                                                           -\sqrt{2} \\
                                                                                                                           1 \\
                                                                                                                    \end{array}
                                                                                                                    \right]\right)$}.
\item The eigenpairs of $A(G_3)$ are {\small $\left(\frac{1+\sqrt{17}}{2},\left[\begin{array}{c}
                                                                     \frac{1+\sqrt{17}}{4}\\
                                                                     1 \\
                                                                     \frac{1+\sqrt{17}}{4}\\
                                                                     1
                                                                     \end{array}
                                                              \right]\right)$, $\left(0,\left[\begin{array}{c}
                                                                                                     -1 \\
                                                                                                      0\\
                                                                                                      1 \\
                                                                                                      0 \\
                                                                                             \end{array}
                                                                                             \right]\right)$,
                                                              $\left(\frac{1-\sqrt{17}}{2},\left[\begin{array}{c}
                                                                                                 \frac{1-\sqrt{17}}{4}\\
                                                                                                 1 \\
                                                                                                 \frac{1-\sqrt{17}}{4}\\
                                                                                                 1
                                                                                                 \end{array}
                                                                                                \right]\right)$} and
                                                              {\small $\left(-1,\left[\begin{array}{c}
                                                                                      0\\
                                                                                     -1 \\
                                                                                      0\\
                                                                                      1 \\
                                                                              \end{array}
                                                                              \right]\right)$}.
\end{itemize}
Let $G=\bigvee_{(H,\mathcal{S})}\mathcal{F}$, where $H=P_3$. Then, denoting $\delta_{i,j}=\delta_{i,j}(H)$ and defining the characteristic
vectors of the vertex subsets $S_1, S_2$ and $S_3$ considering their elements by alphabetic order, we obtain:
\begin{eqnarray*}
A(G) &=& \left(\begin{array}{ccc}
          A(G_1)                        & \delta_{1,2}x_{S_1} x^T_{S_2} & \delta_{1,3}x_{S_1} x^T_{S_3}\\
          \delta_{2,1}x_{S_2} x^T_{S_1} & A(G_2)                        & \delta_{2,3}x_{S_2} x^T_{S_3}\\
          \delta_{3,1}x_{S_3} x^T_{S_1} & \delta_{3,2}x_{S_3} x^T_{S_2} & A(G_3)\\
         \end{array}\right)\\
     &=& \left(\begin{array}{ccc}
          A(G_1)                 & \left[\begin{array}{c}
                                                1 \\
                                                1\\
                                                0\\
                                          \end{array}
                                   \right][1\; 0 \; 1]              & \textbf{0}_{3 \times 4}\\
         \left[\begin{array}{c}
                      1 \\
                      0\\
                      1\\
               \end{array}
         \right][1\; 1 \; 0]     & A(G_2)                           & \left[\begin{array}{c}
                                                                                   1 \\
                                                                                   0\\
                                                                                   1\\
                                                                            \end{array}
                                                                      \right][1\; 0 \; 1 \; 1]\\
         \textbf{0}_{4 \times 3} & \left[\begin{array}{c}
                                                1 \\
                                                0 \\
                                                1 \\
                                                1 \\
                                         \end{array}
                                   \right][1\; 0 \; 1]              & A(G_3)\\
        \end{array}\right).
\end{eqnarray*}
According to Theorem~\ref{nom-main-eigenvalues-join}, $\{0, -1\} \subset \sigma(A(G))$. Notice that $S_1 \subseteq V(G_1)$ is $(1,1)$-regular
and thus we are not able to get a conclusion about if the eigenvalue $0$ of $A(G_1)$ is or not an eigenvalue of $A(G)$. On the other hand
$S_2 \subseteq V(G_2)$ is $(0,2)$-regular and $S_3 \subseteq V(G_3)$ is $(2,2)$-regular. In fact,
$$
\sigma(A(G)) = \{4.44999, 1.86239, 0, 0, 0, -1, -1.3822, -1.51442, -3.02546\}.
$$
\end{example}

Consider a graph $H$ of order $p$, a family of graphs $\mathcal{F}=\{G_1, \ldots, G_p\}$, where each graph $G_i$ has  order $n_i$, and a family
of vertex subsets $\mathcal{S}=\{S_1, \ldots, S_p\}$, where for each $i \in \{1, \ldots, p\}$, $S_i \subseteq V(G_i)$. If
$G=\bigvee_{(H,\mathcal{S})}\mathcal{F}$ and $(\lambda,\widehat{u})$ is an eigenpair of $A(G)$, decomposing $\widehat{u}$ such that
$\widehat{u}=\left(\begin{array}{c}
                          u_1 \\
                          u_2\\
                          \vdots \\
                          u_p\\
                          \end{array}
                   \right)$, where each $u_i$ is a subvector of $\widehat{u}$ with $n_i$ components, then $\lambda \widehat{u}=A(G)\widehat{u}$,
that is,
\begin{eqnarray}
\lambda \left(\begin{array}{c}
                           u_1\\
                           u_2\\
                           \vdots\\
                           u_p
                           \end{array}
                    \right) &=& \left(\begin{array}{c}
       A(G_1)u_1 + \left(\sum_{j \ne 1}{\delta_{1,j} x^T_{S_j}u_j}\right)x_{S_1} \\
       A(G_2)u_2 + \left(\sum_{j \ne 2}{\delta_{2,j} x^T_{S_j}u_j}\right)x_{S_2}\\
       \vdots \\
       A(G_p)u_p + \left(\sum_{j \ne p}{\delta_{p,j} x^T_{S_j}u_j}\right)x_{S_p}\\
       \end{array}
\right),\label{marca-z}
\end{eqnarray}
where $\delta_{i,j}=\delta_{i,j}(H)$.

Furthermore, if we assume that $G_i$ is $d_i$-regular and $S_i$ or its complement is $(k_i,\tau_i)$-regular, for $i=1, \ldots, p$, respectively,
according to Theorem~\ref{nom-main-eigenvalues-join},
$$
\bigcup_{i=1}^{p}{\left(\sigma(G_i) \setminus \{d_i,k_i-\tau_i\}\right)} \subset \sigma(G),
$$
since by one hand, as it is well known, all the eigenvalues of each graph $G_i$ are non-main but $d_i$, on the other hand, if a regular graph
has a $(k,\tau)$-regular vertex subset, then $k-\tau$ is a non-main eigenvalue \cite{CSZ2008}.\\

Additionally, assuming that $S_i=V(G_i)$, for $i=1, \ldots, p$, the remaining eigenvalues of $G$ can be computed as follows: let us define
$\widehat{u}$, setting each of its subvectors $u_i=\theta_ie_{n_i}$, for $i=1, \ldots, p$, where each $e_{n_i}$ is an all one vector with $n_i$
componentes and $\theta_1, \ldots, \theta_p$ are scalars. Then the system \eqref{marca-z} becomes
\begin{eqnarray*}
\lambda \left(\begin{array}{c}
                     \theta_1 e_{n_1} \\
                     \theta_2 e_{n_2}\\
                     \vdots \\
                     \theta_p e_{n_p}\\
         \end{array}\right) &=& \left(\begin{array}{c}
                                      \left(d_1\theta_1 + \sum_{j \ne 1}{\delta_{1,j} \theta_j n_j}\right)e_{n_1} \\
                                      \left(d_2\theta_2 + \sum_{j \ne 2}{\delta_{2,j} \theta_j n_j}\right)e_{n_2}\\
                                      \vdots \\
                                      \left(d_p\theta_p + \sum_{j \ne p}{\delta_{p,j} \theta_j n_j}\right)e_{n_p}\\
                                      \end{array}
                                \right).
\end{eqnarray*}
Therefore, $(\lambda,\widehat{u})$ is an eigenpair for $A(G)$ if and only if $(\lambda,\widehat{\theta})$,
where $\widehat{\theta} = (\theta_1, \theta_2, \ldots, \theta_p)^T$, is an eigenpair of the matrix
\begin{equation}\label{matrix_m}
M = \left(\begin{array}{cccc}
           d_1             & \delta_{1,2}n_2 & \ldots & \delta_{1,p}n_p \\
           \delta_{2,1}n_1 & d_2             & \ldots & \delta_{2,p}n_p \\
           \vdots          & \vdots          & \ddots & \vdots \\
           \delta_{p,1}n_1 & \delta_{p,2}n_2 & \ldots & d_p \\
           \end{array}
    \right),
\end{equation}
that is, $M\widehat{\theta}=\lambda\widehat{\theta}$.\\
Setting $D=\text{diag}(d_1, \ldots, d_p)$ and $N=\text{diag}(n_1, \ldots, n_p)$, then  $M=A(H)N + D$
is similar to the symmetric matrix
\begin{equation}\label{matrix_m2}
M' = \left(\begin{array}{cccc}
           d_1                       & \delta_{1,2}\sqrt{n_1n_2} & \ldots & \delta_{1,p}\sqrt{n_1n_p}\\
           \delta_{2,1}\sqrt{n_1n_2} & d_2                       & \ldots & \delta_{2,p}\sqrt{n_2n_p}\\
           \vdots          & \vdots                              & \ddots & \vdots \\
           \delta_{p,1}\sqrt{n_1n_p} & \delta_{p,2}\sqrt{n_2n_p} & \ldots & d_p \\
           \end{array}
    \right),
\end{equation}
since $M'=KMK^{-1}$ with $K=\text{diag}(\sqrt{n_1}, \ldots, \sqrt{n_p})$. Therefore, $M'=D+KA(H)K$ and $\sigma(M)=\sigma(M')$.

Based on the above analysis, we are able to deduce the following result.

\begin{theorem}\label{teorema2}
Consider a graph $H$ of order $p$ and a family of regular graphs $\mathcal{F}=\{G_1, \ldots, G_p\}$, where each regular graph $G_i$
has degree $d_i$ and order $n_i$. Consider the family of vertex subsets $\mathcal{S}=\{S_1, \ldots, S_p\}$, where
$$
S_i \in \{S'_i \subseteq V(G_i): S'_i \text{ or } V(G_i) \setminus S'_i \text{ is } (k_i,\tau_i)-\text{regular, for some } k_i, \tau_i\},
$$
for $i=1, \ldots, p$. Assume that $G=\bigvee_{(H,\mathcal{S})}\mathcal{F}$ and $M'$ is the matrix defined in \eqref{matrix_m2}.
If $S_i=V(G_i)$, for $i=1, \ldots, p$, then
$$
\sigma(G) = \left(\bigcup_{i=1}^{p}{\sigma(G_i)\setminus \{d_i\}}\right) \cup \sigma\left(M'\right),
$$
otherwise $\sigma(G) \supseteq \bigcup_{i=1}^{p}{\sigma(G_i)\setminus \{d_i, k_i-\tau_i\}}$.
\end{theorem}

\begin{proof}
The conclusions are direct consequence of the above analysis, taking into account that if $S_i=V(G_i)$, for $i =1, \ldots, p$,
then each $S_i$ is $(k_i,\tau_i)$-regular, with $k_i=d_i$ and $\tau_i=0$.
\end{proof}

\section{Some applications on the spread of graphs}

\subsection{Definitions and basic results}

Given a $n \times n$ complex matrix $M$, the spread of $M$, $s(M)$, is defined as $\max_{i,j}|\lambda_i(M)-\lambda_j(M)|$, where the maximum is
taken over all pairs of eigenvalues of $M$. Then
$$
s(M) = \max_{x,y}\left(x^*Mx - y^*My\right) = \max \sum_{i,j}{m_{i,j}\left(\bar{x}_ix_j - \bar{y}_iy_j\right)},
$$
where $z^*$ is the conjugate transpose of $z$ and the maximum is taken over all pairs of unit vectors in $\mathbb{C}^n$.

\begin{theorem}\label{Th-mirsky}\cite{M1956}
$s(M) \le \left(2\sum_{i,j}{|m_{i,j}|^2 - \frac{2}{n}|\sum_{i}{m_{i,i}}|^2}\right)^{1/2}$, with equality if and
only if $M$ is a normal matrix (that is, such that $M^*M=MM^*$), with $n - 2$ of its eigenvalues all equal to the
average of the remaining two.
\end{theorem}

Several results on the spread of normal and Hermitian matrices were presented in \cite{JKW1985, NT1994}.

In this paper, we consider only the spread of adjacency matrices of simple graphs and we define the spread of a graph $G$ as the spread
$s(A(G))$, which is simply denoted by $s(G)$. Therefore,
$$
s(G) = \max_{i,j}\{|\lambda_i(G)-\lambda_j(G)|\},
$$
where the maximum is taken over all pairs of eigenvalues of the adjacency matrix of $G$. If the graph $G$ has order $n$, then
$s(G) =\lambda_1(G) - \lambda_n(G)$ and replacing the matrix $M$ of Theorem~\ref{Th-mirsky} by $A(G)$, it follows that
\begin{equation}
s(G) = \lambda_1(G) - \lambda_n(G) \le \sqrt{4|E(G)|}. \label{mirsky-inequalities1}
\end{equation}
Denoting the average degree of the vertices of $G$ by $\overline{d}(G)$, from \eqref{mirsky-inequalities1} it follows that
\begin{equation}
s(G) \le \sqrt{2n\overline{d}(G)} < \sqrt{2n(n-1)} \label{mirsky-inequalities2}
\end{equation}
if $n>2$, since $\overline{d}(G) \le n-1$ and $\overline{d}(G) = n-1$ if and only if $G=K_n$. Notice that  $\sigma(K_n)=\{n-1,(-1)^{[n-1]}\}$
and then $s(K_n)=n$. Furthermore,
\begin{equation}
\overline{d}(G) \le \frac{n}{2} \Rightarrow s(G) \le n. \label{mirsky-inequalities3}
\end{equation}

In \cite{GHK2001} the following upper bounds on the spread of a graph were obtained.

\begin{theorem}\cite{GHK2001}
If $G$ is a graph of order $n$, then
\begin{equation}
s(G) \le \lambda_1(G) + \sqrt{2|E(G)|-\lambda^2_1(G)} \le 2\sqrt{|E(G)|}. \label{gregory-inequalities}
\end{equation}
Equality holds throughout if and only if equality holds in the first inequality; equivalently, if and only if $|E(G)|=0$
or $G=K_{p,q}$, for some $p$ and $q$.
\end{theorem}

\begin{theorem}\label{kresult}\cite{GHK2001}
If $G$ is a regular graph of order $n$, then $s(G) \le n$. Equality holds if and only if the complement of $G$, $\overline{G}$,
is disconnected.
\end{theorem}

Additional results on the spread of graphs can be found in \cite{GHK2001, LiuMu2009}.

\subsection{The spread of the join of two graphs}

Now it is worth to recall the join of two vertex disjoint graphs $G_1$ and $G_2$ which is the graph $G_1 \vee G_2$ obtained from
their union connecting each vertex of $G_1$ to each vertex of $G_2$.  Considering this graph operation, as direct consequence of
Theorem~\ref{teorema2}, we have the following corollaries. Notice that Corollary~\ref{cor1} is well known (see, for instance,
\cite{Schwenk74}).

\begin{corollary}\label{cor1}
If $G_i$ is a $d_i$-regular graph of order $n_i$, for $i=1, 2$, then
$$
\sigma(G_1 \vee G_2) = \bigcup_{i=1}^{2}\left(\sigma(A(G_i)) \setminus \{d_i\}\right) \cup \{\beta_1,\beta_2\}.
$$
where $\beta_1$ and $\beta_2$ are eigenvalues of the matrix $M' = \left(\begin{array}{cc}
                                                                        d_1           & \sqrt{n_1n_2} \\
                                                                        \sqrt{n_1n_2} & d_2 \\
                                                                        \end{array}
                                                                     \right)$, that is,
\begin{eqnarray}
\beta_1 &=& \frac{d_1 + d_2 + \sqrt{(d_1-d_2)^2+4n_1n_2}}{2} \label{beta1}\\
\beta_2 &=& \frac{d_1 + d_2 - \sqrt{(d_1-d_2)^2+4n_1n_2}}{2} \label{beta2}.
\end{eqnarray}
\end{corollary}


\begin{corollary}\label{cor2}
Consider a $d_i$-regular graph of order $n_i$, for $i=1, 2$,  and the graph $G=G_1 \vee G_2$ of order $n=n_1+n_2$. Then
\begin{eqnarray}
s(G) & = & \left\{\begin{array}{ll}
                \sqrt{(d_1-d_2)^2+4n_1n_2}, & \hbox{if } \lambda_n(G)=\beta_2\\
                \frac{d_2-d_1 + \sqrt{(d_1-d_2)^2+4n_1n_2}}{2}+s(G_1), & \hbox{if } \lambda_n(G)=\lambda_{n_1}(G_1)\\
                \frac{d_1-d_2 + \sqrt{(d_1-d_2)^2+4n_1n_2}}{2}+s(G_2), & \hbox{if } \lambda_n(G)=\lambda_{n_2}(G_2).
                \end{array}
              \right.\label{marca}
\end{eqnarray}
Furthermore, setting $R=\sqrt{(d_1-d_2)^2+4n_1n_2}$,
$$
s(G) = \max \{R, \frac{d_2-d_1 + R}{2}+s(G_1), \frac{d_1-d_2 + R}{2}+s(G_2)\}.
$$
\end{corollary}

\begin{proof}
According to Corollary~\ref{cor1}, $\sigma(A(G)) = \bigcup_{i=1}^{2}\left(\sigma(A(G_i)) \setminus \{d_i\}\right) \cup \{\beta_1,\beta_2\},$
where $\beta_1$ and $\beta_2$ have the values \eqref{beta1} and \eqref{beta2}, respectively. On the other hand, $\lambda_{n_i}(G_i)=d_i-s(G_i)$,
for $i=1,2$. Therefore, the equalities in \eqref{marca} follows, as well as the second part.
\end{proof}

\begin{corollary}\label{cor3}
Let $G_i$ be a $d_i$-regular graph of order $n_i$, for $i=1,2$, and $G = G_1 \vee G_2$. If $|d_1 - d_2| > |n_1 - n_2|$, then
$s(G) > n = n_1+n_2$.
\end{corollary}

\begin{proof}
By construction, it is immediate that the order of $G$ is $n=n_1+n_2$. Taking into account that $\beta_1$ and $\beta_2$, in \eqref{beta1}
and \eqref{beta2} of Corollary~\ref{cor1}, respectively, are eigenvalues of the adjacency matriz of $G = G_1 \vee G_2$, then
$$
s(G) \ge \beta_1 - \beta_2 = \sqrt{(d_1-d_2)^2+4n_1n_2} > n_1 + n_2 = n.
$$
Notice that
$\sqrt{(d_1-d_2)^2+4n_1n_2} > n_1 + n_2 \Leftrightarrow (d_1-d_2)^2+4n_1n_2 > n_1^2 + n_2^2 + 2n_1n_2 \Leftrightarrow (d_1-d_2)^2 > (n_1-n_2)^2$.
\end{proof}

Considering the complete graph graph $K_k$, for which $\sigma(K_k)=\{(-1)^{[k-1]},k-1\}$, and the null graph $\overline{K}_{n-k}$,
for which e $\sigma(\overline{K}_{n-k})=\{0^{[n-k]}\}$, and denote the join of these graphs by $G(n, k)$ (that is
$G(n, k)=K_k \vee \overline{K}_{n-k}$), according to Corollary~\ref{cor1}, $\sigma(G(n, k)) = \{(-1)^{[k-1]},0^{[n-k-1]},\beta_1,\beta_2\}$,
with
\begin{eqnarray*}
\beta_1 &=& \frac{k-1 + \sqrt{(k-1)^2+4k(n-k)}}{2}\\
\beta_2 &=& \frac{k-1 - \sqrt{(k-1)^2+4k(n-k)}}{2}.
\end{eqnarray*}
Therefore, $s(G(n, k)) = \beta_1 - \beta_2 = \sqrt{(k-1)^2+4k(n-k)}$. Furthermore, when $\frac{n+1}{3} < k < n-1$, the  hypothesis of
Corollary~\ref{cor3} hold for these graphs and then $s(G(n, k))>n$.


\begin{theorem}\cite{GHK2001}\label{spread1}
Among the family of graphs $G(n, k)= K_k \vee \overline{K}_{n-k}$, with $1 \le k \le n-1$, the maximum of $s(G(n, k))$ is attained
when $k=\lfloor 2n/3 \rfloor$.
\end{theorem}

In \cite{GHK2001} the following conjecture was checked by computer for  graphs of order $n \le 9$.

\begin{conjecture}\cite{GHK2001}
The maximum spread $s(n)$ of the graphs of order $n$ is attained only by $G(n,\lfloor 2n/3 \rfloor)$, that is,
$s(n) = \lfloor(4/3)(n^2-n+1)\rfloor^{1/2}$ and so $\frac{1}{\sqrt{3}}(2n-1) < s(n) < \frac{1}{\sqrt{3}}(2n-1) + \frac{\sqrt{3}}{4n-2}$.
\end{conjecture}

\subsection{The spread of the generalized join of graphs}
Throughout this subsection we consider a graph $H$ of order $p$ and a family of regular graphs $\mathcal{F}=\{G_1, \ldots, G_p\}$, where
each regular graph $G_i$ has degree $d_i$ and order $n_i$. We consider also $M=A(H)N + D$, where $N=\text{diag}(n_1, \ldots, n_p)$ and
$D=\text{diag}(d_1, \ldots, d_p)$, and we define $d_{i^*} - s(G_{i^*}) = \min \{d_i - s(G_i): i=1, \ldots, p\}$ and the matrix
$$
P = \left(\begin{array}{cccc}
            0             & \sqrt{n_1n_2} & \ldots & \sqrt{n_1n_p} \\
           \sqrt{n_1n_2}  & 0             & \ldots & \sqrt{n_2n_p} \\
           \vdots         & \vdots        & \ddots & \vdots \\
           \sqrt{n_1n_p}  & \sqrt{n_2n_p} & \ldots & 0 \\
           \end{array}
    \right).
$$
Using the above notation, with the following theorems, we state upper and lower bounds on the spread of $G=\bigvee_{H}\mathcal{F}$.

\begin{theorem}\label{FinalResultTh}
If $G=\bigvee_{H}\mathcal{F}$, then
\begin{equation}\label{FinalResult1}
s(G) = s(M) + \max_{1 \le i \le p}\{\lambda_p(M) + s(G_i) - d_i,0\}.
\end{equation}
Furthermore,
\begin{equation}\label{FinalResult2}
s(G) \ge n_{\downarrow}\left(s(H) - ( \widetilde{d}_{\uparrow} - \widetilde{d}_{\downarrow})\right)
         - (n_{\uparrow} - n_{\downarrow})\left(\lambda_p(H) + \widetilde{d}_{\uparrow}\right)
\end{equation}
where $\widetilde{d}_{\uparrow} = \max_{1 \le i \le p}{\frac{d_i}{n_i}}$ ($\widetilde{d}_{\downarrow} = \min_{1 \le i \le p}{\frac{d_i}{n_i}}$),
and $n_{\uparrow} = \max_{1 \le i \le p}{n_i}$ ($n_{\downarrow} = \min_{1 \le i \le p}{n_i}$).
\end{theorem}

\begin{proof}
According to Theorem~\ref{teorema2}, $\sigma(G) = \left(\bigcup_{i=1}^{p}{\sigma(G_i)\setminus \{d_i\}}\right) \cup \sigma(M).$ Then
$\forall i \in \{1, \ldots, p\} \; \lambda_{n_i}(G_i) = d_i - s(G_i) \in \sigma(G)$ and hence
$$
\lambda_n(G) \in \{d_i - s(G_i), i=1, \ldots, p\} \cup \{\lambda_n(M)\}.
$$
Since $\lambda_1(G)=\lambda_1(M)$ (notice that $\lambda_1(G) \ge d_i \; \forall i \in \{1, \ldots, p\}$),
the equality \eqref{FinalResult1} holds.\\
Now, we prove the inequality \eqref{FinalResult2}. Consider the symmetric matrix $M'=KA(H)K+D$ in \eqref{matrix_m2}, where
$K=\text{diag}(\sqrt{n_1}, \ldots, \sqrt{n_p})$, which is similar to the matrix $M$. Let $\left(\lambda,x\right)$ be an eigenpair of $H,$
where $x$ is such that $\sum_{i=1}^{p}{x_i^2}=1$. Setting $y=K^{-1}x$, then
\begin{eqnarray*}
\lambda_n(G) & \le & \min \sigma(M) = \min \sigma(M')\\
             & \le & \frac{y^T \left(K A(H)K+D\right)y}{y^Ty}\\
             &  =  & \frac{x^TA(H)x + x^T K^{-1}DK^{-1}x}{x^T K^{-2}x}\\
             &  =  & \frac{\lambda x^Tx + x^TDN^{-1}x}{\sum_{i=1}^{p}{\frac{x_i^2}{n_i}}} \\
             &  =  & \frac{\lambda + \sum_{i=1}^{p}{\frac{d_i}{n_i}x_i^2}}{\sum_{i=1}^{p}{\frac{1}{n_i}x_i^2}} \le \lambda_1(M') \le \lambda_1(G).
\end{eqnarray*}
Taking into account that $\widetilde{d}_{\uparrow} = \max_{1 \le i \le p}{\frac{d_i}{n_i}}$
($\widetilde{d}_{\downarrow} = \min_{1 \le i \le p}{\frac{d_i}{n_i}}$) and $n_{\uparrow} = \max_{1 \le i \le p}{n_i}$
($n_{\downarrow} = \min_{1 \le i \le p}{n_i}$), we may conclude the following.
\begin{itemize}
\item If $\lambda=\lambda_p(H)$, then
      $\lambda_n(G) \le  \frac{\lambda_p(H) + \widetilde{d}_{\uparrow}}{\frac{1}{n_{\uparrow}}} = n_{\uparrow}\left(\lambda_p(H) +
                         \widetilde{d}_{\downarrow}\right).$
\item If $\lambda=\lambda_1(H)$, then
      $\lambda_1(G) \ge  \frac{\lambda_1(H) + \widetilde{d}_{\downarrow}}{\frac{1}{n_{\downarrow}}} = n_{\downarrow}\left(\lambda_1(H) +
                        \widetilde{d}_{\uparrow}\right).$
\end{itemize}
Therefore, $s(G) \ge n_{\downarrow}\left(\lambda_1(H) + \widetilde{d}_{\downarrow}\right) - n_{\uparrow}\left(\lambda_p(H) +
\widetilde{d}_{\uparrow}\right) = n_{\downarrow}\left(\lambda_1(H) + \widetilde{d}_{\downarrow}\right) - n_{\downarrow}\left(\lambda_p(H) +
              \widetilde{d}_{\uparrow}\right) - (n_{\uparrow} - n_{\downarrow})\left(\lambda_p(H) + \widetilde{d}_{\uparrow}\right).$
\end{proof}

As immediate consequence of Theorem~\ref{FinalResultTh}, we have the following corollary.

\begin{corollary}
If the graph $H$ has at least one edge and $G=\bigvee_{H}\mathcal{F}$, then
$$
s(G) \ge n_{\downarrow}\left(s(H) - (\widetilde{d}_{\uparrow} - \widetilde{d}_{\downarrow})\right).
$$
\end{corollary}

\begin{proof}
From \eqref{FinalResult2}, it follows
\begin{eqnarray}
s(G) & \ge & n_{\downarrow}\left(s(H) - ( \widetilde{d}_{\uparrow} - \widetilde{d}_{\downarrow})\right)
             - (n_{\uparrow} - n_{\downarrow})\left(\lambda_p(H) + \widetilde{d}_{\uparrow}\right) \nonumber \\
     & \ge & n_{\downarrow}\left(s(H) - (\widetilde{d}_{\uparrow} - \widetilde{d}_{\downarrow})\right) \label{marca1}
\end{eqnarray}
The inequality \eqref{marca1} is obtained taking into account that $\widetilde{d}_{\uparrow} \le 1$ and, since $H$
has at least one edge, $\lambda_p(H) \le -1$ and therefore,
$(n_{\uparrow} - n_{\downarrow})\left(\lambda_p(H) + \widetilde{d}_{\uparrow}\right) \le 0$.
\end{proof}

Using this corollary, and taking into account that $\widetilde{d}_{\uparrow}$ and $\widetilde{d}_{\downarrow}$ are both
in the interval $(0,1)$, it follows that $s(G) \ge n_{\downarrow}(s(H)-1)$.

\begin{theorem}
If $G=\bigvee_{H}\mathcal{F}$, then
$$
s(G) \le \max_{1 \le i \le p}{d_i}+\lambda_1(H)\lambda_1(P) - \min \{d_{i^*} - s(G_{i^*}),\lambda_p(M)\}.
$$
\end{theorem}

\begin{proof}
By Theorem~\ref{teorema2}, $\sigma(G) = \left(\bigcup_{i=1}^{p}{\sigma(G_i)\setminus \{d_i\}}\right) \cup \sigma(M),$ where
$M=D + A(H) \circ P$, with $D=\text{diag}(d_1, \ldots, d_p),$ and $\circ$ denotes the Hadamard product (see, for instance, \cite{HJ1991}).
Since when we have two symmetric nonnegative matrices of order $p$, $A$ and $B$, $\lambda_1(A+B) \le \lambda_1(A) + \lambda_1(B)$
and $\lambda_1(A \circ B) \le \lambda_1(A \otimes B) =\lambda_1(A)\lambda_1(B)$, where $\otimes$ is the Kronecker product, we may conclude
that
$$
\lambda_1(M) \le \lambda_1(D)+\lambda_1(A(H) \circ P) \le \lambda_1(D)+\lambda_1(H)\lambda_1(P)=\max_{1 \le i \le p}{d_i}+\lambda_1(H)\lambda_1(P).
$$
Since $\lambda_n(G)=\min \{d_{i^*} - s(G_{i^*}),\lambda_p(M)\}$, it follows that,
$$
s(G) \le \max_{1 \le i \le p}{d_i}+\lambda_1(H)\lambda_1(P) - \min \{d_{i^*} - s(G_{i^*}),\lambda_p(M)\}.
$$
\end{proof}

\begin{theorem}
If the graph $H$ has at least one edge and $G=\bigvee_{H}\mathcal{F}$, then
$$
s(M) \le s(G)  < s(M) + \max_{1 \le \le p} \{d_i\}.
$$
\end{theorem}

\begin{proof}
By Theorem~\ref{FinalResultTh}, $s(G) = s(M) + \max_{1 \le i \le p}\{\lambda_p(M)-\lambda_{n_i}(G_i),0\}$.
\begin{enumerate}
\item If $\max_{1 \le i \le p}\{\lambda_p(M)-\lambda_{n_i}(G_i),0\} = 0,$ then the left inequality holds as equality and the right inequality is strict.
\item Otherwise, assume that $\exists i^* \in \{1, \ldots, p\}$ such that
      $\max_{1 \le i \le p}\{\lambda_p(M)-\lambda_{n_i}(G_i),0\} = \lambda_p(M) - \lambda_{n_i^*}(G_{i^*})$. Since,
      $$
      \lambda_p(M) - \lambda_{n_i^*}(G_{i^*}) <  - \lambda_{n_i^*}(G_{i^*}) \le d_{i^*} \le \max_{1 \le i \le p}\{d_i\},
      $$
      then the right inequality holds. Notice that, when $H$ has at least one edge, $\lambda_p(M)<0$. In fact, if $ij \in E(H)$, the matrix
      $B_{ij} = \left(\begin{array}{cc}
                             d_i           & \sqrt{n_in_j} \\
                             \sqrt{n_in_j} & d_j \\
                      \end{array}
                \right)$ is a principal submatrix of $P_{ij}MP^T_{ij}$, where $P_{ij}$ is permutation matrix. Therefore,
      $\lambda_p(M) = \lambda_p(P_{ij}MP^T_{ij}) \le \lambda_2(B_{ij}) < 0$.
      The left inequality follows from the fact that the eigenvalues of $M$ are also eigenvalues of $G$.
\end{enumerate}
\end{proof}
\subsection{An infinite family of non regular graphs of order $n$ with spread equal to $n$.}

\begin{theorem}\label{families_with_spread_n}
Consider the positive integres $p, q \ge 3$ and $n \in \mathbb N$ such that $n \ge p+q+3$. Let $H=P_3$ and let $\mathcal{F}=\{G_1, G_2, G_3\}$
be a family of graphs, where $G_1=C_p$, $G_2=C_q$ and $ G_3= C_{n-p-q}$. If $\mathcal{S}=\{S_1, S_2, S_3\}$ is such that $S_i = V(G_i)$ for
$i=1,2,3$, then the graph
\begin{equation}
G=\bigvee_{(H,\mathcal{S})}\mathcal{F}.\label{generalized_join_family}
\end{equation}
is non regular and is such that $s(G) \le n$. Furthermore, $s(G)=n$ if and only if $q=\frac{n}{2}$.
\end{theorem}

\begin{proof}
By definition of generalized join, it is immediate that $G$ is non regular. By Theorem~\ref{teorema2}
$$
\sigma(G) =  \bigcup_{i=1}^{3}\left(\sigma(G_i) \setminus \{2\}\right) \cup \{\beta_1,\beta_2, \beta_3\},
$$
where $\beta_i,$ with $i\in \{1,2,3\},$ are the roots of the characteristic polynomial of the matrix
$$
M = \left(\begin{array}{ccc}
            2&     q    &     0 \\
            p &    2    &  n-p-q \\
            0 &    q    & 2
          \end{array}
    \right).
$$
Then $\beta_1 = 2, \beta_2=2+ \sqrt{q(n-q)},$ and $\beta_3=2- \sqrt{q(n-q)}.$ Notice that the largest eigenvalue of $M$ is
$\beta_2$ and $\lambda_{min}(G) = \beta_3=2- \sqrt{q(n-q)} < -2$ (taking into account the values of $p$, $q$ and $n$
and since $\lambda_{min}(G_i) \ge -2$, for $i=1, 2, 3$). Therefore,
$$
s(G) = 2 \sqrt{q(n-q)}.
$$
Since $q(n-q) \le \frac{n^2}{4}$ and $q(n-q) = \frac{n^2}{4}$ if and only if $q=\frac{n}{2}$, the result follows.
\end{proof}

As immediate consequence of Theorem~\ref{families_with_spread_n}, if $n$ is an even positive number not less than $12$,
$q=\frac{n}{2}$ and $3 \le p \le \frac{n-6}{2}$ then the graph $G$ defined in \eqref{generalized_join_family} is such
that $s(G)=n$.


\begin{thebibliography}{99}

\bibitem{CFMR2011}
D. M. Cardoso, M. A. de Freitas, E. A. Martins, M. Robbiano, Spectra of graphs obtained by a generalization of the join graph
operation, manuscript submitted for publication.

\bibitem{CSZ2008}
D. M. Cardoso, I. Sciriha, C. Zerafa, Main eigenvalues and $(k,\tau)$-regular sets. Linear Algebra Appl. 423 (2010): 2399-2408.

\bibitem{CR2007}
D. M. Cardoso, P. Rama, Spectral results on regular graphs with $(k,\tau)$-regular sets.  Discrete Math. 307 (2007): 1306-1316 .



\bibitem{GHK2001}
D. A. Gregory, D. Hershkowitz, S. J. Kirkland, The spread of the spectrum of a graph, Linear Algebra Appl. 332-334 (2001): 23-35.

\bibitem{HJ1991}
R. A. Horn, C.R. Johnson, Topics in matrix analysis, Cambridge University Press, New York, 1991.

\bibitem{JKW1985}
C.R. Johnson, R. Kumar, H. Wolkowicz, Lower bounds for the spread of a matrix, Linear Algebra Appl. 71 (1985): 161–173.

\bibitem{LiuMu2009}
B. Liu, Liu Mu-huo, On the spread of the spectrum of a graph, Linear Algebra Appl. 309 (2009): 2727-2732.


\bibitem{M1956}
L. Mirsky, The spread of a matrix, Mathematika 3 (1956): 127-130.

\bibitem{NT1994}
P. Nylen, T.-Y. Tam, On the spread of a Hermitian matrix and a conjecture of Thompson, Linear and Multilinear Algebra 37 (1994) 3–11.

\bibitem{Schwenk74}
A. J. Schwenk, Computing the characteristic polynomial of a graph, Graphs and Combinatorics (Lecture notes in Mathematics 406, eds.
R. Bary and F. Harary), Springer-Verlag, Berlin-Heidelberg-New York, 1974, pp. 153-172.

\end{thebibliography}
\end{document}